\documentclass[reqno,11pt]{amsart}

\IfFileExists{mymtpro2.sty}{%
  \usepackage[subscriptcorrection]{mymtpro2}
}{}

\usepackage{a4,amssymb}
\usepackage{mathrsfs}
\usepackage{graphicx}

\newtheorem{theorem}{Theorem}
\newtheorem{proposition}[theorem]{Proposition}
\newtheorem{definition}[theorem]{Definition}
\newtheorem{lemma}[theorem]{Lemma}

\theoremstyle{definition}

\newcommand{\bel}{\begin{equation} \label}
\newcommand{\ee}{\end{equation}}
\newcommand{\R}{\mathbb{R}}

\newcommand{\C}{\mathbb{C}}

\theoremstyle{remark}

\newtheorem{myremarks}[theorem]{Remarks}

    {\end{nummer}\end{myremarks}}

\newcounter{numcount}
\newcommand{\labelnummer}{\mbox{\normalfont (\roman{numcount})}}%

\makeatletter

\newenvironment{nummer}%
  {\let\curlabelspeicher\@currentlabel%
    \begin{list}{\labelnummer}%
      {\usecounter{numcount}\leftmargin0pt%
        \topsep0.5ex\partopsep2ex\parsep0pt\itemsep0ex\@plus1\p@%
        \labelwidth2.5em\itemindent3.5em\labelsep1em%
      }%
    \let\saveitem\item%
    \def\item{\saveitem%
      \def\@currentlabel{{\upshape\curlabelspeicher}$\,$\labelnummer}}%
    \let\savelabel\label%
    \def\label##1{\savelabel{##1}%
      \@bsphack%
        \ifmmode\else%
          \protected@write\@auxout{}%
          {\string\newlabel{##1item}{{\labelnummer}{\thepage}}}%
        \fi%
      \@esphack%
    }%
  }{\end{list}}%

\renewcommand{\appendix}{\def\thesection{\textsc{Appendix}}}

\renewcommand\d{\mathrm d}

 \let\leq\le
 \let\geq\ge

 \let\Im\undefined

\DeclareMathOperator{\Im}{Im}

\DeclareMathOperator{\tr}{tr\kern1pt}

\renewcommand{\Im}{\mathop\mathrm{Im}\nolimits}

\newcommand{\N}{{\mathbb N}}

\makeatletter

\newif\ifper\pertrue
\def\per{.}

\def\bti{\@ifnextchar[\bbti\bbbti}
\def\bbti[#1]#2{#2, #1.}
\def\bbbti#1{#1.}

\def\z{\@ifnextchar[\zz\zzz}
\def\zz[#1]#2#3#4#5{\perfalse\emph{#2} \textbf{#3}, #4 (#5) [#1]}
\def\zzz#1#2#3#4{\emph{#1} \textbf{#2}, #3 (#4)\ifper\per\fi\pertrue}

\def\pub{\@ifstar\pubstar\pubnostar}
\def\pubnostar{\@ifnextchar[\@@pubnostar\@pubnostar}
\def\@@pubnostar[#1]#2#3#4{#2, #3, #4, #1\ifper\per\fi\pertrue}
\def\@pubnostar#1#2#3{#1, #2, #3\ifper\per\fi\pertrue}
\def\pubstar[#1]#2#3#4{\perfalse #2, #3, #4 [#1]\pertrue}

\makeatother

\newcommand{\beq}{\begin{equation}}
\newcommand{\eeq}{\end{equation}}
\newcommand{\ba}{\begin{array}}
\newcommand{\ea}{\end{array}}
\newcommand{\bea}{\begin{eqnarray}}
\newcommand{\eea}{\end{eqnarray}}
\newcommand{\beas}{\begin{eqnarray*}}
\newcommand{\eeas}{\end{eqnarray*}}

{

\def\P{I\kern-.30em{P}}
\def\E{I\kern-.30em{E}}
\renewcommand{\E}{\mathbb{E}\mkern2mu}
\renewcommand{\P}{\mathbb{P}}

\begin{document}

\title[Iso-resonant potentials]{Compactness of iso-resonant potentials for Schr\"odinger operators in dimensions one and three}

\author[P.\ D.\ Hislop]{Peter D.\ Hislop}
\address{Department of Mathematics,
    University of Kentucky,
    Lexington, Kentucky  40506-0027}
\email{peter.hislop@uky.edu}

\author[R.\ Wolf]{Robert Wolf}
\address{
    Department of Mathematical Sciences,
    University of Cincinnati,
    Cincinnati, OH 45221-0025}
\email{robert.wolf@uc.edu}

\thanks{Both authors were partially supported by NSF grant DMS 11-03104 during the time this work was done. This paper is partly based on the dissertation submitted by the second author in partial fulfillment of the requirements for a PhD at the University of Kentucky. PDH thanks P.\ Aubin, T.\ J.\ Christiansen, and P.\ A.\ Perry for several conversations on resonances.}

\begin{abstract}
We prove compactness of a restricted set of real-valued, compactly supported potentials $V$ for which the corresponding Schr\"odinger operators $H_V$ have the same resonances, including multiplicities. More specifically, let $B_R(0)$ be the ball of radius $R > 0$ about the origin in $\R^d$, for $d=1,3$. Let $\mathcal{I}_R (V_0)$ be the set of real-valued potentials in $C_0^\infty( \overline{B}_R(0); \R)$ so that the corresponding Schr\"odinger operators have the same resonances, including multiplicities, as $H_{V_0}$. We prove that the set $\mathcal{I}_R (V_0)$ is a compact subset of $C_0^\infty (\overline{B}_R(0))$ in the $C^\infty$-topology. An extension to Sobolev spaces of less regular potentials is discussed.
\end{abstract}

\maketitle \thispagestyle{empty}


\tableofcontents

\vspace{.2in}



\section{Statement of the problem and results}\label{sec:introduction}
\setcounter{equation}{0}

We are interested in the resonances of Schr\"odinger operators
$H_V = - \Delta + V$, acting on $L^2 ( \R^d)$, for $d=1,3$,
with compactly supported, real-valued potentials $V \in C_0^\infty (\R^d)$.
Resonances are the poles of the meromorphic continuation of the cut-off resolvent
of $H_V$ defined as follows.
For any $V \in C_0^\infty ( \R^d )$, we denote by $\chi_V \in C_0^\infty ( \R^d )$ any
smooth, real-valued, compactly supported function
such that $\chi_V V = V$. For $\Im \lambda > 0$, the cut-off resolvent $R_V(\lambda) := \chi_V (H_V - \lambda^2)^{-1} \chi_V$ is meromorphic operator-valued function with at most finitely-many
poles on the positive imaginary axis. These poles at $i \lambda_j$, with $\lambda_j > 0$, are distinguished by the fact that
$-\lambda_j^2$ is a negative eigenvalue of $H_V$.
This operator-valued function $R_V(\lambda)$ admits a meromorphic continuation as a bounded operator on $L^2 (\R^d)$ across the real axis to the lower half-complex plain when $d \geq 1$ is odd.
The poles of the continuation are the resonances of $H_V$. They are independent of the choice of cut-off $\chi_V$ satisfying the above conditions.

Our goal is to characterize the set of potentials so that the corresponding Schr\"odinger operators are iso-resonant, that is, they have the same resonances with the same multiplicities. However, if $V$ and $\tilde{V}$ are related by a Euclidean symmetry, the Schr\"odinger operators $H_V$ and $H_{\tilde{V}}$ are iso-resonant.
Consequently, we must remove translational invariance from the problem to get a meaningful result.
We fix a ball $B_R(0)$ of finite radius $R >0$ and consider a fixed potential $V_0 \in C_0^\infty (\overline{B}_R(0); \R)$.
The \textbf{iso-resonant set} of $V_0$, denoted here by $\mathcal{I}_R(V_0)$, is the set of all $V \in C_0^\infty (\overline{B}_R(0))$ so that $H_V$ has the same resonances as $H_{V_0}$, including multiplicities. This set is invariant under the action of the compact rotation group.
Our main result is the following theorem.

\begin{theorem}\label{thm:iso-res-cmpt1}
 The set $\mathcal{I}_R (V_0)$ of iso-resonant potentials to $V_0 \in C_0^\infty (\overline{B}_R(0))$ is compact in the $C^\infty$-topology in dimensions $d=1, 3$.
\end{theorem}

Theorem \ref{thm:iso-res-cmpt1} is the analog of the compactness result for iso-spectral potentials of Br\"uning \cite{bruning}. He considered a compact Riemannian manifold $(M,g)$ and Schr\"odinger operators $H_V = -\Delta_g + V$, where $-\Delta_g \geq 0$ is the Laplace-Beltrami operator on $M$ for metric $g$, and the real-valued potential $V \in C^\infty (M)$. These self-adjoint, lower semi-bounded operators $H_V$ have only discrete spectrum accumulating at infinity. For a fixed metric $g$, given $V_0 \in C^\infty (M)$, the  \textbf{iso-spectral set} of $V_0$ is the set of real-valued potentials $V \in C^\infty (M)$ so that the corresponding Schr\"odinger operators have the same eigenvalues, including multiplicity, as $H_{V_0}$.
For $\dim M = 2,3$, Br\"uning proved that the iso-spectral set of $V_0$ is compact in the $C^\infty$-topology.
The restriction of the dimension comes from use of the Sobolev Embedding Theorem, Theorem \ref{thm:sobolevEmbedding1}. This result was improved and extended by Donnelly \cite{donnelly}.

The existence of infinitely-many resonances for Schr\"odinger with nonzero, real-valued potentials $V \in C_0^\infty (\R^d; \R)$  in odd dimensions greater than or equal to three was proven first by Melrose \cite{melrose} using the Poisson formula and wave invariants discussed in section \ref{sec:poisson1}. This result was extended for more general potentials, including super-exponentially decaying potentials of noncompact support, by S\'a Barreto and Zworski \cite{sabarreto-zworski1,sabarreto-zworski2}.
The existence of resonances for Schr\"odinger operators in even dimensions $d \geq 4$ was proven by S\'a Barreto and Tang \cite{sabarreto-tang1}. We note that these results prove the resonance-rigidity of the zero potential. Since $H_0 = - \Delta$ has no resonances, the iso-resonance set of $V_0 = 0$ contains precisely one element among super-exponentially decaying real-valued potentials.
This rigidity is also a consequence of the analysis of the heat invariants in section \ref{sec:poisson1}.

The present work focuses on real-valued potentials.
Christiansen \cite{tjc1,tjc2} proved that there are many complex-valued potentials for which the Schr\"odinger operator $H_V$ has no resonances. This implies that among complex-valued potentials, the iso-resonance set of $V_0 = 0$ is large. In fact, Christiansen \cite{tjc1} constructs a large family of nontrivial complex-valued potentials with no resonances.

In a recent work \cite{smith-zworski1}, Smith and Zworski relaxed the smoothness assumption and proved that a nontrivial real-valued potential in $L^\infty_0 (B_R(0); \R)$, for $d \geq 3$ odd, has at least one resonance. Furthermore, they proved that if two potentials $V_j \in
L^\infty_0 (\R^d; \R)$, for $j=1,2$, are iso-resonant, then $V_1 \in H^m (\R^d) ~\mbox{if and only if} ~V_2 \in H^m (\R^d)$, for any $m \in \N$. We can then define iso-resonant classes for less regular potentials.
If $V_0 \in H^m (B_R(0); \R)$, for $m \geq 2$, then $V_0 \in L_0^\infty ( B_R(0); \R)$ since $d \leq 3$.

We define $\mathcal{I}_R^m(V_0) \subset H^m (B_R(0))$ as the iso-resonant class of $V_0$ in $H^m(B_R(0))$.
Although the Poisson formula \eqref{eq:wave-trace2} holds for the wave trace if $V \in L^\infty_0(\R^d)$,
we do not know if a finite term small $t$-expansion holds for the wave trace as Smith and Zworski \cite{smith-zworski1} proved for the heat trace (see section \ref{sec:less-reg-compactness} for a discussion.)
As a consequence, we must make an assumption:

\vspace{.1in}

\noindent
{\bf{WTE}}. \emph{For $V \in H^m (\overline{B}_R(0))$, with $m \geq 2$ and $d=3$, the wave trace admits a small $t$ asymptotic expansion of the form
\beq\label{eq:wave-trace-less1}
{\rm Tr}(W_V(t)-W_0(t)) \sim  w_1 (V) \delta(t) + \sum_{j=2}^{m+2} w_j (V)|t|^{2j-3} + r_{m+2}(t) ,
\eeq
where $r_{m+2}(t) \in C^{2m + 1} (\R)$.}

\vspace{.1in}

\noindent
We prove the following result on the compactness of less regular iso-resonant potentials.

\begin{proposition}\label{prop:main2}
We suppose the dimension $d=1,3$.
\begin{enumerate}
\item For any $V_0 \in C_0^\infty (\overline{B}_R(0); \R)$, if a potential $V \in L^\infty_0 (\overline{B}_R(0); \R)$ is iso-resonant to $V_0$, then $V \in C_0^\infty (\overline{B}_R(0); \R)$.
Hence, the iso-resonant set of $V_0$ in $L^\infty_0 (\overline{B}_R(0); \R)$ is the same as the iso-resonant set of $V_0$ in
$C^\infty_0 (\overline{B}_R(0); \R)$.
\item If $V_0 \in H^m (\overline{B}_R(0); \R)$ for $m \geq 3$, $d=3$, and (WTE) holds, then $\mathcal{I}_R^m(V_0)$ is compact in the $H^{m-3}$-topology.
\end{enumerate}
\end{proposition}

There seem to be few results on the characterization of iso-resonant potentials.
The one-dimensional problem was studied by Zworski \cite{zworski1}. He proved that if $V_0 \in L_0^\infty (\R; \R)$ and is even, the poles of the $S$-matrix determine $V_0$ uniquely if $H_{V_0}$ does not have a half-bound state (a resonance) at zero. If $V_0$ has a half-bound state at zero, Zworski proved that there exists exactly one other potential $V_1 \neq V_0$ whose $S$-matrix has the same poles. The poles of the meromorphic continuation of
the $S$-matrix are contained in the set of resolvent resonances defined here. Other results on the inverse problem for Schr\"odinger operators on $\R$ may be found, for example, in Korotyaev \cite{korotyaev1}, Bledsoe \cite{bledsoe1}, and Bennewitz, Brown, and Weikard \cite{bbw1} and references therein.

Datchev and Herazi \cite{datchev-hezari1} considered the iso-resonant
problem for smooth, positive, real, super-exponentially decaying {\it radial} potentials $V(x) = R(|x|)$ in dimension $d \geq 1$ odd in the semiclassical regime. Additionally, the potentials are monotonic with $R^\prime(r=0)=0$ and at any other point where $R(r) = 0$. They proved that if two such potentials $V_0$ and $V$ have the same resonance set up to $o(h^2)$, then there exists a vector $x_0 \in \R^d$ so that $V(x) = V_0(x-x_0)$. This shows that the resonances uniquely determine, in the semiclassical sense, the radial potential up to translations.

There are some results on iso-resonant and iso-phasal
metrics for the Laplace-Beltrami operator on various families of noncompact Riemannian manifolds.
Hassell and Zworski \cite{hassell-zworski1} proved uniqueness for obstacle scattering in three-dimensions. If a bounded, connected obstacle $\Omega \subset \R^3$ has the same resonances as a ball of the same volume,
then the region is such a ball. This proves resonance rigidity of the ball in $\R^3$.

We mention that Theorem \ref{thm:iso-res-cmpt1} should hold for $d=2$. However, the Poisson formula of Zworski \cite{zworski2} is not strong enough to allow us to prove that the wave invariants are constant across $\mathcal{I}_R(V_0)$ from the small $t$-asymptotics of the wave trace.
In recent work \cite{tjc3}, Christiansen proved (among other results) that the heat invariants are constant across the iso-resonant class $\mathcal{I}_R(V_0)$ for Schr\"odinger operators in even dimensions $d \geq 2$ using different methods. As a consequence, Christiansen proves \cite[Theorem 1.3]{tjc3} compactness of the iso-resonant class $\mathcal{I}_R(V_0)$ for the $d=2$ case.


\subsection{Idea of the proof and contents}\label{subsec:contents}

The idea of the proof of Theorem \ref{thm:iso-res-cmpt1} is as follows. The Poisson formula for the regularized trace of the wave group shows that any potential $V \in \mathcal{I}_R (V_0)$ has the same regularized wave trace as $V_0$ since they have the same resonances, including multiplicities.
From the small time asymptotic expansion of the regularized wave trace, it follows that the coefficients of powers of $t$ in the expansion are the same for any $V \in \mathcal{I}_R (V_0)$. The coefficients in the small $t$-expansion of the wave trace are known constant multiples of the heat coefficients occurring in the small time expansion
of the regularized heat trace. These heat invariants are integrals of $V$ and its derivatives.
Using the machinery developed by Br\"uning \cite{bruning} and by Donnelly \cite{donnelly} for the iso-spectral case, extended to our noncompact case, the equality of these coefficients across the iso-resonance set imply that families of Sobolev norms are uniformly bounded on the set $\mathcal{I}_R (V_0)$.
We then prove that this implies compactness of $\mathcal{I}_R (V_0)$ in the $C^\infty$-topology using the uniform bounds on the Sobolev norms and the regularized determinant.

We review basic facts concerning Schr\"odinger operators and their resonances in section \ref{sec:resonances1} including the meromorphic continuation of the resolvent. We define the regularized determinant and introduce the analytic functions whose zeros coincide with the resonances of $H_V$. We prove a continuity result for these functions. The Poisson formula is described in section \ref{sec:poisson1}. Compactness of $\mathcal{I}_R(V_0)$ for smooth potentials is proved in section \ref{sec:compactness}. The case of less regular potentials is presented in section \ref{sec:less-reg-compactness}.


\subsection{Notation}\label{subsec:notation}

For an open subset $\Omega \subset \R^d$, with boundary satisfying the cone condition, see Theorem \ref{thm:sobolevEmbedding1}, we denote the associated Sobolev spaces by $H^{s,p} (\Omega)$ and their norms by $\| \cdot \|_{s,p}$. When $s=0$, we write $H^{0,p} (\Omega)$ as $L^p (\Omega)$, and we write $\| \cdot \|_p$ for $\| \cdot \|_{0,p}$. When $p=2$, we write $H^{s,2} (\Omega)$ as $H^s (\Omega)$, and we write the norm as
$\| \cdot \|_{s,2}$. We also write $\| \cdot \|_\infty$ for the $L^\infty$-norm.


\section{Resonances of Schr\"odinger operators}\label{sec:resonances1}
\setcounter{equation}{0}

In this section, we recall the characterization of the resonances of a Schr\"odinger operator $H_V$
as the zeros of a holomorphic function and present some continuity results for this functions
with respect to the potential. The resonances of $H_V$ are defined via the meromorphic extension of
the cut-off resolvent $\mathcal{R}_V ( \lambda )
\equiv \chi_V ( H_V - \lambda^2 )^{-1} \chi_V$ of $H_V$.
This truncated resolvent may be expressed
in terms of the cut-off resolvent for the Laplacian $H_0 = - \Delta$, defined as $\mathcal{R}_0 ( \lambda )
\equiv \chi_V ( H_0 - \lambda^2 )^{-1} \chi_V$. 
The cut-off resolvent $\mathcal{R}_0 ( \lambda )$ of the Laplacian $H_0 = - \Delta$
is a holomorphic, bounded, operator-valued function
on $L^2 ( \R^d)$ for $\lambda \in \C$ for $d \geq 1$ odd,
and for $\lambda \in \Lambda$, for $d \geq 4$ even. For $d=2$, the cut-off free resolvent is
holomorphic on $\Lambda \backslash \{ 0 \}$, with a logarithmic singularity at $\lambda = 0$.

An application of the second resolvent formula
allows us to write the cut-off perturbed resolvent $\mathcal{R}_V(\lambda)$ for $H_V$ as
\beq\label{eq:cutoffresolv1}
\mathcal{R}_V(\lambda )( 1 + V \mathcal{R}_0 ( \lambda)) = \mathcal{R}_0 (\lambda) .
\eeq
The resolvent $R_V(\lambda)$ is analytic for $\Im \lambda >> 0$, and it is meromorphic
for $\Im \lambda > 0$ with at most finitely-many
poles with finite multiplicities corresponding to the eigenvalues of $H_V$.
Thanks to (\ref{eq:cutoffresolv1}) and the meromorphic Fredholm Theorem (see, for example, \cite[Theorem XIII.13]{reed-simon4}), the cut-off perturbed resolvent
$\mathcal{R}_V(\lambda)$ extends as
a meromorphic, bounded operator-valued function on $\C$ if $d \geq 1$ is odd, and onto the
Riemannn surface $\Lambda$ if $\d \geq 4$ is even, with an additional logarithmic
singularity at $\lambda = 0$ when $d = 2$.
The poles of these continuations are the resonances of $H_V$. They are independent of the choice of
$\chi_V$ satisfying the above conditions.


From \eqref{eq:cutoffresolv1}, the resonances of the Schr\"odinger operator $H_V$ are the zeros of
the meromorphic continuation of the bounded operator
$1 + V \mathcal{R}_0(\lambda) = 1 + V \chi_V R_0(\lambda) \chi_V$
to the lower-half complex plane, or the Riemann surface $\Lambda$, depending on the parity of the dimension. It is convenient to introduce a holomorphic function whose zeros correspond with these zeros.
Since for $d \geq 2$, the operator $K_V(\lambda) := V \mathcal{R}_0(\lambda)$ is no longer trace class,
it is necessary to use the $p$-regularized determinant (see, for example,
\cite[Appendix B.7]{dyatlov-zworski2016}).

We denote the $p^{th}$ von Neumann-Schatten trace ideal of operators by $\mathcal{L}_p$.
A bounded operator $A \in \mathcal{L}_p$ if the singular values $\mu_j(A)$ of $A$
satisfy $\sum_j \mu_j(A)^p < \infty$.

\begin{definition}\label{defn:regularized-det1}
For any operator $A\in \mathcal{L}_p$, with $p \geq 1$,
we define the operator $R_p(A)$ by
\beq\label{eq:regularization1}
R_p(A) := (I+A)\exp\left(\sum\limits_{i=1}^{p-1}\frac{(-A)^i}{i!}\right) - 1.
\eeq
\end{definition}

It is easy to check that $R_p(A)$ is a trace class operator. As a consequence, the regularized determinant of $A$
may be defined using $R_p(A)$.

\begin{definition}\label{defn:regualrized-det2}
For any operator $A\in \mathcal{L}_p$, with $p \geq 1$, the regularized $p$-determinant of $A$ is defined to be
\beq\label{eq:p-det1}
\det_p(I+A):= \det(I+R_p(A)) = \det \left[ (I+A)\exp\left(\sum\limits_{j=1}^{p-1}\frac{(-A)^j}{j!}\right) \right] .
\eeq
\end{definition}

\noindent
Applying these definitions to the resolvent of $H_0$, we have the basic result.

\begin{proposition}\label{prop:reg-det1}
Let $K_V(\lambda) := V R_0(\lambda) \chi_V$, on $L^2 (\R^d)$, denote the extension of the bounded operator
$V R_0(\lambda) \chi_V$ for $\Im \lambda >> 0$ to $\C$ or $\Lambda$, depending on the dimension, described above.
\begin{enumerate}
\item The operator $K_V(\lambda) : = V R_0(\lambda) \chi_V$
belongs to the trace ideal $\mathcal{L}_p$, with $p=1$ when $d =1$, $p=2$ for $d=2,3$, and, in general,
$K_V(\lambda) \in \mathcal{L}_p$ for $p > (d+1)/2$ when $d > 3$.

\item The regularized $p$-determinant of $K_V(\lambda)$,
$$
D_{p,V}(\lambda) := \det_p ( 1 + K_V(\lambda)) = \det( 1 + R_p(K_V(\lambda))),
$$
extends to an entire function on $\C$ or $\Lambda$, depending on the parity of $d$.

\item The zeros of $D_{p,V}(\lambda)$
occur at values $\lambda$ for which $\lambda^2$ is an eigenvalue or resonance of $H_V$. Moreover,
the multiplicities of the zeros of $D_{p,V}(\lambda)$ equal the (algebraic) multiplicities of the eigenvalues or resonances of $H_V$.
\end{enumerate}
\end{proposition}

The proof of this standard theorem may be found in \cite[section 3.4]{dyatlov-zworski2016}. For simplicity,
we will occasionally omit the index $p$ from the notation for the regularized determinant and simply write $D_V(\lambda)$.

%

The multiplicity $m_{V}(\lambda)$ of a pole of $V\mathcal{R}_{0}(\lambda)$ at $\lambda$ is defined as follows (see, for example,  \cite[section 4.2]{dyatlov-zworski2016}).
The multiplicity of any nonzero resonance $\lambda_0$ of $R_V(\lambda)$ is given by
$$
m_V(\lambda_0) := {\rm Rank} \left[ \int_{\gamma_0} R_V(\lambda) ~2 \lambda ~ d\lambda \right] ,
$$
where $\Gamma_0$ is a simple closed contour about $\lambda_0$ containing no other pole.
The multiplicity at zero is slightly more complicated as there may be a resonance and an eigenvalue at zero:
$$
m_V(0) := \frac{1}{2} {\rm Rank} \left[ \int_{\gamma_{0}} R_V(\lambda) ~ d\lambda \right] + {\rm Rank} \left[ \int_{\gamma_{0}} R_V(\lambda)~2 \lambda ~ d\lambda \right].
$$
If the second integral vanishes, and the first one does not, there is only a resonance at zero energy.
We can now define the resonance set of $H_V$.

\begin{definition}\label{defn:ResonanceSet1}
The \emph{resonance set} $\mathcal{R}(V)$ of $H_V$ be the set of $\lambda \in \C$ consisting
 of values $i \lambda$, $\lambda > 0$, so that $- \lambda^2$ is an eigenvalue (necessarily nonpositive) of $H_V$,
 and all $\lambda \in \C^-$ that are poles of $\mathcal{R}_V (\lambda)$, including multiplicities $m_V(\lambda)$: $\mathcal{R}(V) := \{ ( \lambda, m_V(\lambda)) ~|~ m_V(\lambda) \neq 0 \}$.
\end{definition}

Let $m_{D}(\lambda)$ be the order of zero of $D_{V}(\lambda)$ at $\lambda$. The equality of these two multiplicities is proved in Dyatlov and Zworski \cite[Theorem 8.3.3]{dyatlov-zworski2016}:
$$
m_{V}(\lambda) = m_{D_V}(\lambda).
$$
This relation plays an important role in our compactness argument.

The continuity properties of $D_V(\lambda)$ with respect to the potential $V$ is important in order to prove the compactness of $\mathcal{I}_R (V_0)$.


\begin{proposition}\label{prop:conv-mult1}
Suppose $V_j \in L_0^\infty(B_R(0))$ converges to $V_\infty$ in the $L^\infty$-norm. 
Then the analytic functions
$D_{2, V_j}(\lambda)$ converge uniformly on compact subsets of $\C$ to the analytic function $D_{2, V_\infty}(\lambda)$.
\end{proposition}

\begin{proof}
1. We give the proof for $d=3$.
Let $\{V_i\}\subset	L^\infty (B_R(0))$ be a sequence of potentials converging in the $L^\infty$-topology
to $V_\infty$. We let
\beq\label{eq:sup-norm1}
\nu_0 := \max \{ \sup_j \| V_j\|_\infty, \|V_\infty \|_\infty \}
\eeq
denote the uniform bound on the family of potentials.
Let $K_i (\lambda) := V_i R_0 (\lambda) \chi_{R}$, where $\chi_R$ is the characteristic function on $B_R(0)$ and note that $\chi_R V_i = V_i$, and similarly for $K_\infty (\lambda)$. We also write $K_0(\lambda) := \chi_R R_0(\lambda) \chi_R$. We note that these operators $K_X(\lambda)$, for $X =0, i, \infty$, are in the Hilbert-Schmidt class for any $\lambda \in \C$ and $\| K_X(\lambda) \|_2 \leq C(\lambda)$. For $d=3$,
we recall that $R_{2, V_i}(\lambda) := (I+K_i (\lambda))e^{-K_i(\lambda)}-I \in \mathcal{I}_1$, so that
by a standard trace ideal identity (see, for example \cite[Theorem 3.4]{simon})
\bea\label{eq:determ-diff1}
|D_{V_i} (\lambda) -D_{V_\infty}(\lambda)| & = & | \det(I+R_{2,V_i}(\lambda))- \det(I+R_{2,V_\infty} (\lambda))| \nonumber  \\
 & \leq & \| R_{2,V_i} (\lambda) -R_{2,V_\infty}(\lambda) \|_1 \exp \{1+\|R_{2,V_i}(\lambda)\|_1 + \|R_{2,V_\infty}(\lambda)\|_1 \} \nonumber \\
& \leq &  C_{V_0, \lambda} \|(I+K_i(\lambda))e^{-K_i(\lambda)}-(I+K_\infty(\lambda)) e^{-K_\infty (\lambda)}\|_1.
\eea
It follows from the convergence estimates proved below that the function
$$
C_{\nu_0, \lambda} := \exp \{1+\|R_{2,V_i}(\lambda)\|_1 + \|R_{2,V_\infty}(\lambda)\|_1 \}
$$
is locally uniformly bounded in $\lambda$ and depends only on $\nu_0$ defined in \eqref{eq:sup-norm1}.
Upon expanding the exponentials on the right side of \eqref{eq:determ-diff1}, we have
\beq\label{eq:determ-diff2}
(I+K_i(\lambda))e^{-K_i(\lambda)}-(I+K_\infty(\lambda))e^{-K_\infty(\lambda)}
= \sum\limits_{m=2}^\infty\frac{(m-1)(-1)^m}{m!}[K_\infty^m(\lambda) - K_i^m(\lambda)].
\eeq

\noindent
2. We consider the terms $K_\infty^m(\lambda) - K_i^m (\lambda)$ and factor the difference as
\beq\label{eq:trace-est0}
K_i^m(\lambda) - K_\infty^m(\lambda) = \sum_{\ell=1}^m K_i^{\ell -1}(\lambda)(K_i (\lambda) - K_\infty(\lambda))K_\infty^{m- \ell}(\lambda).
\eeq
To compute the trace norm of the summands in \eqref{eq:trace-est0}, we consider two cases: $m \geq 3$ and $m=2$.
For $m \geq 3$ and $\ell \geq 3$, we estimate a typical term on the right in \eqref{eq:trace-est0} as
\bea\label{eq:eq:trace-est1}
\lefteqn{ \|K_i^{\ell-1}(\lambda)(K_i (\lambda) - K_\infty(\lambda)) K_\infty^{m - \ell}(\lambda)\|_1 } \nonumber \\
& \leq & \| V_i - V_\infty \|_\infty  \|K_i^{\ell-1}(\lambda)\|_1 \| K_0(\lambda) K_\infty^{m- \ell}(\lambda)\| \nonumber  \\
&\leq &  \|V_i - V_\infty \|_\infty \|K_i (\lambda)^2 \|_1 \| K_i (\lambda) \|^{\ell - 3} \| K_0(\lambda) K_\infty (\lambda)^{m - \ell}\|
\nonumber \\
&\leq & \|V_i - V_\infty \|_\infty  \|K_i (\lambda) \|_2^2 \| K_0(\lambda) \| \|K_i (\lambda)\|^{\ell-3}\|K_\infty(\lambda)\|^{m-\ell} .  \nonumber \\
 & &
\eea
For $m \geq 3$ and $\ell =2$, the bound is
$$
\|V_i - V_\infty \|_\infty  \|K_i (\lambda)\| \|K_0(\lambda) \|_2 \| K_\infty (\lambda)\|_2  \|K_\infty(\lambda)\|^{m-3},
$$
and for  $m \geq 3$ and $\ell = 1$, we obtain
$$
\|V_i - V_\infty \|_\infty  \|K_0 (\lambda)\| \|K_0(\lambda) \|_2 \| K_\infty (\lambda)\|_2  \|K_\infty(\lambda)\|^{m-2}.
$$
Similarly, for the case of $m=2$, we find the bound of
\beq\label{eq:trace-est2}
\|V_i - V_\infty \|_\infty [ \|K_i (\lambda)\|_2 + \| K_\infty (\lambda) \|_2] \|K_0(\lambda) \|_2 .
\eeq
%
%
%
%
%
As a consequence of \eqref{eq:trace-est0}--\eqref{eq:trace-est2}, and using the uniform bound on $\|K_X(\lambda)\|$ given in \eqref{eq:resolvent-est1}, we obtain
\bea
\|K_i^m (\lambda) - K_\infty^m (\lambda) \|_1  &  \leq  & \sum\limits_{l=1}^m \|K_i (\lambda)^{l-1}(K_\infty (\lambda) -K_i (\lambda) )K_\infty (\lambda)^{m-l}\|_1 \nonumber  \\
&\leq & \sum\limits_{l=1}^m  m \|V_i - V_\infty\|_\infty\langle C_{V_0} \lambda\rangle^m e^{-c(Im \lambda)m}.
\eea

\noindent
3. Returning to the main estimate of the trace norm of \eqref{eq:determ-diff2}, we obtain
\bea\label{eq:determ-diff3}
\lefteqn{
\|(I+K_i (\lambda))e^{-K_i(\lambda)}-(I+K_\infty (\lambda))e^{-K_\infty (\lambda)}\|_1} \nonumber \\
  & \leq &    \sum\limits_{m=2}^\infty\frac{(m-1)}{m!}\|(K_i(\lambda))^m-(K_\infty(\lambda))^m\|_1 \nonumber \\
  & \leq  & \| V_i - V_\infty \|_\infty \left( \sum\limits_{m=2}^\infty\frac{(m-1)me^{-c(Im \lambda)m}}{m!} \right) \nonumber \\
 & \leq & C(| \Im \lambda |, \nu_0)  \| V_i - V_\infty \|_\infty.
 \eea
The ratio test shows that the sum converges for all $\lambda \in \C$.
The constant $C (| \Im \lambda |, \nu_0) > 0$ is locally uniformly bounded on compact subsets of $\C$ and
depends on the uniform bound $\nu_0$ \eqref{eq:sup-norm1} on family of potentials $V_j$. Since $\|V_i - V_\infty\|_\infty\to 0$, estimates \eqref{eq:determ-diff1} and \eqref{eq:determ-diff3} imply
the locally uniform convergence of the functions $D_{V_i}(\lambda)$ to $D_{V_\infty}(\lambda)$.
\end{proof}

We will apply this proposition to a sequence of potentials in $C_0^\infty ( \overline{B}_R(0); \R)$ in the proof of  Theorem \ref{thm:iso-res-cmpt1}, and to a sequence of potentials in $H^m( \overline{B}_R(0))$ for Proposition \ref{prop:main2}.


\section{Poisson formula for the wave trace and the wave invariants}\label{sec:poisson1}
\setcounter{equation}{0}

In this section, we recall the Poisson formula for the wave trace (see, for example \cite[Part 1, chapter 3]{dyatlov-zworski2016} or \cite[chapter 4]{melrose} and its connection with the resonances. We then derive a Gilkey-type formula for the wave invariants using a representation of the heat invariants due to Hitrik and Polterovich \cite{hitrik-polterovich}.

\subsection{Poisson formula}\label{subsec:poisson-form1}

The Poisson formula connects the regularized wave trace with the resonance set $\mathcal{R}(V)$ of $H_V$.
We briefly recall the formulation of the wave trace, for further details see, for example, \cite[Chapter 4]{melrose}. The operators $H_0$ and $H_V$ are associated with initial value problems for the wave equation on $\R^d$. We write $H_X$ for $H_0$ or $H_V$. Let $u(x,t)$ denote the solution of the initial-value problem for the wave equation:
\beq\label{eq:wave-eqn1}
(\partial_{tt} - H_X )u(x,t) = 0, ~~~u(x,t=0) = u_0(x), ~~~\partial_t u(x,t)|_{t=0} = u_1(x).
\eeq
Letting $w(x,t) := ( u(x,t) , (\partial_t u)(x,t))^{\rm T}$, it is easy to check that $w(x,t)$ solves the system of equations
\beq\label{eq:wave-eqn2}
-i \frac{\partial}{\partial t} w(x,t) = \mathcal{L}_X w(x,t),
\eeq
with initial conditions $w(x,0) = u_0(x), u_1(x))$. The operator
$\mathcal{L}_X$ is the matrix-valued operator on the Hilbert space $H^1(\R^d) \oplus L^2 (\R^d)$
given by
\beq\label{eq:wave-eqn3}
\mathcal{L}_X := \left( \begin{array}{cc}
 0     & 1 \\
  H_X & 0
  \end{array} \right), ~~~ X=0, V.
\eeq
The wave group $W_X(t)$ associated with the Schr\"odinger operator $H_X$ is the matrix-valued unitary operator
\beq\label{eq:wave-gp1}
W_X(t) := e^{it \mathcal{L}_X }.
\eeq
It implements the time evolution $w(x,t) = W_X(t) w(x, t=0)$.
The associated regularized wave trace is the distribution defined by
\beq\label{eq:wave-gp2}
  {\rm Tr} (W_V(t) - W_0(t)) = 2 {\rm Tr} (\cos t H_V^{1/2} - \cos t H_0^{1/2} ), ~~ t \neq 0.
\eeq

It is well-known that the regularized wave trace is a distribution, even in $t \neq 0$, whose explicit formula depends on the resonance set $\mathcal{R}(V)$, the eigenvalues and resonances of $H_V$.
We define the \emph{resonance set} $\mathcal{R}(V)$ of $H_V$ be the set of complex numbers consisting
 of values $i \lambda$, $\lambda > 0$, so that $- \lambda^2$ is an eigenvalue (necessarily nonpositive) of $H_V$,
 and all $\lambda \in \C^-$ that are poles of $\mathcal{R}_V (\lambda)$, including multiplicities: $\mathcal{R}(V) := \{ ( \lambda, m_V(\lambda)) ~|~ m_V(\lambda) \neq 0 \}$.
In the case of odd dimensions, the wave trace is related to the resonance set by the following Poisson formula.

\begin{theorem}\label{thm:poisson}\cite[Proposition 4.2]{melrose}
Let $d \geq 3$ be odd, and let $\mathcal{R}(V)$ be the resonance set of $H_V$.
Let $W_X(t)$ be the wave group for $H_X$, for $X = 0, V$. Then, as distributions,
$$
{\rm Tr} (W_V(t) - W_0(t)) = \sum_{\lambda \in \mathcal{R}(V)} ~m_V(\lambda) e^{i |t|  \lambda}, ~~~t \neq 0 .
$$
\end{theorem}

The small time asymptotics of the regularized wave trace in odd dimensions $d \geq 3$ are given by
\beq\label{eq:wave-trace1}
{\rm Tr}(W_V(t)-W_0(t)) \sim \sum\limits_{j=1}^{\frac{d-1}{2}}w_j (V) D^{n-1-2j} \delta(t) + \sum\limits_{j=\frac{d+1}{2}}^N w_j (V)|t|^{2j-d} + r_N(t) ,
\eeq
where $r_N(t) \in C^{2N - d} (\R)$.
We note that in dimension $d=3$, the expansion of the wave trace has the form
\beq\label{eq:wave-trace2}
{\rm Tr}(W_V(t)-W_0(t)) \sim  w_1 (V) \delta (t) + \sum\limits_{j=2}^{N} w_j (V) |t|^{2j-3} + r_N(t) .
\eeq
These formulas may be found in \cite[Lemma 4.1]{melrose}.


It is known that the wave invariants $w_j (V)$ are dimension-dependent multiples of the heat invariants associated to the pair $(H_0, H_V)$, see, for example, \cite[Proposition 2.3]{sabarreto-zworski2}. The heat invariants are the coefficients occurring in the small time asymptotic expansion of the regularized heat trace
\beq\label{eq:heat-trace1}
{\rm Tr} ( e^{-t H_V} - e^{-t H_0} ) \sim \frac{1}{(4 \pi t)^{\frac{d}{2}}} ~\sum_{j=1}^\infty c_j (V) t^j .
\eeq
S\'a Barreto and Zworski \cite{sabarreto-zworski2} noted that the basic formula
$$
e^{-t x^2} = \frac{1}{(4 \pi t)^{\frac{1}{2}}} \int ~ e^{- \frac{s^2}{4 t}} ~\cos (s x) ds,
$$
may be used to express the regularized heat trace \eqref{eq:heat-trace1}
in terms of an integral of the regularized wave trace. Substituting the asymptotic expansion \eqref{eq:wave-trace2}
into this integral results in the following identities relating the heat invariants to the wave invariants:
\bea\label{eq:wave-to-heat-coef1}
w_j (V) & = & \begin{cases}
      \frac{2^{2(j-d)+1}}{M_j}c_j (V) & 1 \leq j\leq \frac{d-1}{2} \\
      \frac{2^{2(j-d)+1}}{N_j}c_j (V) & j\geq \frac{d+1}{2}
   \end{cases} \nonumber \\
   & := & d_j c_j (V) .
\eea
The constants $N_j$ and $M_j$ are given by
\beq\label{eq:wave-to-heat-coef2}
\begin{array}{cccc}
M_j & = & \left[ \left( \frac{d}{d\theta} \right)^{d-1-2j} e^{-\theta^2}\right]_{\theta=0}, & 1 \leq j \leq \frac{d-1}{2}  ,  \\
N_j & =  & \int ~e^{-\theta^2}|\theta|^{2j-d}d\theta, & j \geq \frac{d+1}{2}.
\end{array}
\eeq

We record here well-known \cite[section 4.1]{melrose} expressions for the heat invariants:
\bea\label{eq:heat-inv-list1}
c_1(V) &=& \int_{\R^d} ~V(x) ~d^dx \nonumber \\
c_2(V) &=& \int_{\R^d} ~ V(x)^2 ~d^dx \nonumber \\
c_3(V) &=&  \int_{\R^d} ~\left[ V(x)^3 + \frac{1}{2} | \nabla V(x)|^2 \right] ~d^d x.
\eea


\subsection{Gilkey-type formulas for the wave invariants}\label{subsec:gilkey-formulas-heatinv1}

In this section, we will prove explicit formulas for the heat invariants $c_j(V)$ in terms of the potential $V$ and its derivatives. These will yield formulas for the wave invariants from \eqref{eq:wave-to-heat-coef1}.
For Schr\"odinger operator on a compact manifold, as studied by Br\"uning  \cite{bruning} and by Donnelly \cite{donnelly},
the heat invariants $c_j (V)$, as defined in \eqref{eq:wave-to-heat-coef1}, may be explicitly written in terms of the potential and the metric. Gilkey \cite{gilkey1}
provided a general formula valid for compact manifolds similar to formula \eqref{eq:wave-tr-coef1} for the heat invariants. In the noncompact case of $\R^d$, both Ba\~nuelos and S\'a Barreto \cite{banuelos-sabarreto2} and Hitrik and Polterovich \cite{hitrik-polterovich} proved formulas for the heat invariants $c_j(V)$.
We show that formulas similar to those otained by Gilkey hold for suitable potentials starting with the formulas derived by Hitrik and Polterovich \cite{hitrik-polterovich}.
These authors proved that the heat invariants $c_j (V)$ are obtained by
$$
c_j (V) = \int_{\R^d} ~c_j(x) ~d^dx,
$$
where the densities $c_j(x)$ are given by
\beq\label{eq:heat-coef1}
c_j(x) = (-1)^j  \sum_{k=0}^{j-1} c_{j,k} \frac{(H_V)_y^{k+j} ( \| x-y\|^{2k} ) |_{x=y}}{4^k k! (k+j)!} ,
\eeq
where $(H_V)_y$ denotes the Schr\"odinger operator $H_V$  in the $y$-variable.
The numerical coefficients $c_{j,k}$ are given by
$$
c_{j,k} = \left( \begin{array}{c}
   j-1 + \frac{d}{2} \\
   k + \frac{d}{2}
     \end{array}  \right) .
$$
Due to the relation between the heat and the wave invariants \eqref{eq:wave-to-heat-coef1}, we have the following result.
We define the index set of $k$-tuples $\mathcal{A}_{j,k}$, for $j \geq 3$ and
$3 \leq k \leq j$ as:
\bea\label{eq:indexset1}
\mathcal{A}_{j,k} &=& \left\{ \alpha = ( \alpha^1, \ldots , \alpha^k) ~ \left| ~ \begin{array}{l}
   \alpha^m = ( \alpha_1^m, \ldots, \alpha_d^m) \in \N_0^d \\
   | \alpha^m| = \sum_{\ell = 1}^d \alpha_\ell^m \leq j-k \\
   \sum_{m=1}^k  | \alpha^m| = 2(j-k) \\
   \sum_{m=1}^k \alpha_{\ell}^m \in 2 \N, \forall \ell =1, \ldots, d.
   \end{array}
\right. \right\} \nonumber \\
& &
\eea

\begin{proposition}\label{prop:wave-coef1}
Suppose $V \in C_0^{\infty}(\R^d; \R)$. Then
the wave invariants $w_j (V)$, $j \geq 3$, appearing in the small time asymptotics of the wave trace \eqref{eq:wave-trace1}--\eqref{eq:wave-trace2}, are given by
\beq\label{eq:wave-tr-coef1}
 w_j (V)  = d_j\int_{\R^d} |\nabla^{j-2}V|^2 +\sum_{\alpha \in \mathcal{A}_{j,k}}
 ~ c_\alpha \int_{\R^d} ~ (D^{\alpha^1}V)(D^{\alpha^2}V) \cdots (D^{\alpha^k}V) ~d^d x,
 \eeq
where $d_j$ is the constant defined in \eqref{eq:wave-to-heat-coef1} and the index $\mathcal{A}_{j,k}$ is defined in \eqref{eq:indexset1}.
\end{proposition}

\noindent
As a consequence, a bound on the Sobolev norm of $V$ is obtained from a rearrangement of \eqref{eq:wave-tr-coef1}. For
each $j \geq 3$, we have the bound
\beq\label{eq:wave-tr-coef2}
 \|V\|_{j-2,2}^2 \leq C_j \left( 1+ \sum\limits_{k=3}^j \sum_{\alpha \in \mathcal{A}_{j,k}}
 ~ \int_{\R^d} ~ |(D^{\alpha^1}V) (D^{\alpha^2}V) \cdots (D^{\alpha^k} V)| ~d^dx \right) ,
 \eeq
where $C_j > 0$ depends on $w_j(V)$. In our application, the constant is independent of our choice of $V$ from the iso-resonant set $\mathcal{I}_R(V_0)$, depending only on $j$ and $V_0$. We will utilize this bound in section \ref{sec:compactness} to obtain uniform bounds on the Sobolev norms for $\mathcal{I}_R(V_0)$.


\section{Properties of the iso-resonant set: Uniform bounds on the Sobolev norms}\label{sec:uniform-bounds}
\setcounter{equation}{0}

We will successively bound the higher-order Sobolev norms by induction.

\begin{theorem}\label{thm:uniform-sobolev1}
Let us assume that  $\|V\|_{j-3,2} \leq C$, for some $j \geq 3$.
For $d = 1$, there is a constant $C > 0$ so that
$$
 \|V\|_{j-2,2}^2 \leq C,
 $$
 whereas for $d =3$, there is a constant $C>0$ and $\beta$ with $0 \leq \beta < 2$ so that
 $$
  \|V\|_{j-2,2}^2 \leq C (1+\|V\|^\beta_{j-2,2})
$$
Consequently, if $\| V \|_{0,2}$ is uniformly bounded for $V \in \mathcal{I}_R(V_0)$,
then for each $s \in \N$, there is a finite constant $C_s (V_0) > 0$, depending only on $s$ and $V_0$, so that for all $V \in \mathcal{I}_R(V_0)$,
we have the uniform bound
\beq\label{eq:sobolev-bound1}
\| V \|_{s,2} \leq C_s (V_0).
\eeq
\end{theorem}

The strategy of the proof of Theorem \ref{thm:uniform-sobolev1} is to show that each term
$$
\int_{\R^d} ~ |(D^{\alpha^1}V) (D^{\alpha^2}V) \cdots (D^{\alpha^k}V)| ~d^dx
$$
is bounded by a constant independent of our choice of $V \in \mathcal{I}_R(V_0)$ ($d=1$) or by a multiple of $1+\|V\|^\beta_{j-2,2}$ with $\beta<2$  ($d\geq 3$).  Together these bounds yield a uniform bound on $\|V\|_{j-2,2}$ for each $j$. The details of this argument, similar to those in \cite{bruning} and in Donnelly \cite{donnelly}, are sketched in the appendix in section \ref{sec:sobolev-bounds1}.

We mention, as in \cite{bruning} and in \cite{donnelly}, why the iteration procedure does not work for $d \geq 4$. The heat invariant $c_3(V)$ is given by
\beq\label{eq:third-heat1}
c_3(V) = \int_{\R^d} \left[ V(x)^3 + \frac{1}{2} | \nabla V(x)|^2 \right] ~d^dx.
\eeq
We want to extract the $H^1$-bound on $V$ from this expression. In particular,
We need to prove a bound of the form
\beq\label{eq:v3bound1}
\left| \int_{\R^d} ~V(x)^3 ~d^dx \right| \leq C(\| V \|_2) \| V\|_{1,2}^\beta,
\eeq
for some $0 \leq \beta < 2$.
Using the generalized H\"older inequality, we obtain
\beq\label{eq:v3bound2}
\left| \int_{\R^d} ~V(x)^3 ~d^dx \right| \leq \| V\|_4^2 \|V \|_2.
\eeq

We recall from the Sobolev Embedding Theorem, Theorem \ref{thm:sobolevEmbedding1}, the inequality
\beq\label{eq:sobolev1}
\| u \|_q \leq C \| u \|_{k,p},
\eeq
for indices
$$
\frac{1}{q} = \frac{1}{p} - \frac{k}{d} .
$$
We are interested in $p=2$. In that case, for $k=0$, we have $q=2$ which  is the $L^2$-norm. For $k=1$, we obtain
$q= 2d (d-2)^{-1}$, so for $d=3$, we obtain $q=6$ whereas for $d=4$, we obtain $q=4$. Applying the interpolation result \eqref{thm:interpolation1} in the case $d=3$, this interpolation result has the form:
\beq\label{eq:interpolate1}
\|V\|_4 \leq \|V\|_2^{\frac{1}{4}} ~\|V\|_6^{\frac{3}{4}}.
\eeq
As a consequence, substituting these bounds into the right side of \eqref{eq:v3bound2}, we obtain for $d =3$:
\bea\label{eq:v3bound3}
\left| \int_{\R^d} ~V(x)^3 ~d^dx \right|  &  \leq  & \| V\|_4^2 \|V \|_2 \nonumber \\
 & \leq & \| V \|_2^{\frac{3}{2}} \|V\|_6^{\frac{3}{2}} \nonumber \\
  & \leq & C \| V \|_2^{\frac{3}{2}} \|V\|_{1,2}^{\frac{3}{2}}
\eea

Since the exponent of $\|V\|_{1,2}$ is less than two, and the $L^2$-norm is constant across $\mathcal{I}_R(V_0)$,
the term on the left of \eqref{eq:v3bound3} can be absorbed into the left side of the inequality of \eqref{eq:sobolev-bound1}. Repeating the same analysis for $d=4$ (and similarly for $d \geq 5$),
with the appropriate interpolation estimate \eqref{thm:interpolation1}, we find the bound
\beq\label{eq:v3bound4}
\left| \int_{\R^d} ~V(x)^3 ~d^dx \right|    \leq
C \| V \|_2^{\frac{3}{2}} \|V\|_{1,2}^2,
\eeq
and this cannot be used in the inductive step because the exponent of the term $\|V\|_{1,2}$  is two and thus cannot be absorbed in the left side of \eqref{eq:sobolev-bound1}.



 \section{Compactness of the iso-resonant set $\mathcal{I}_R(V_0)$}\label{sec:compactness}
\setcounter{equation}{0}

In this section, we combine the results of section \ref{sec:poisson1} and \ref{sec:sobolev-bounds1} to prove the main Theorem \ref{thm:iso-res-cmpt1}. We fix a nontrivial, real-valued potential $V_0 \in C_0^\infty (\overline{B}_R(0); \R)$ and recall that $\mathcal{I}_R(V_0)$ is the set of similar potentials iso-resonant with $V_0$.  It follows from the Poisson formula, Theorem \ref{thm:poisson}, and the relation between the wave and the heat invariants \eqref{eq:wave-to-heat-coef1},
that for all $V \in \mathcal{I}_R(V_0)$, we have the equality
\beq\label{eq:heat-inv-eq1}
c_j(V_0) = c_j(V),  j \geq \frac{d+1}{2}, ~~d\geq 3 ~{\rm odd}.
\eeq
For $d=3$, it follows that the equality $c_2(V_0) = c_2(V)$ implies that
$$
\| V \|_{2} = \| V_0 \|_{2}, ~\forall V \in {\mathcal{I}}_R(V_0).
$$
This provides the first step of the induction in Theorem \ref{thm:uniform-sobolev1} so we conclude that for each $s \in \N$, there is a finite constant $C_s(V_0) > 0$ so that $\| V \|_{s,2} \leq C_s(V_0)$ for all $V \in \mathcal{I}_R(V_0)$.

It now remains to prove compactness. Let $\mathcal{C}_s := \{ C_s ~|~ s \in \N  \}$ be any sequence of finite, nonnegative constants. We first consider a larger set $\mathcal{V}_R (\mathcal{C}_s) \subset C_0^\infty (\overline{B}_R(0);\R)$ consisting of all $V \in C_0^\infty (\overline{B}_R(0); \R)$ such that $\| V\|_{s,2} \leq C_s$, for all $s \in \N$.
It is clear that $\mathcal{I}_R(V_0) \subset \mathcal{V}_R( \mathcal{C}_s)$, for a suitable sequence $\mathcal{C}_s$.

We recall the Fr\'echet metric on $C_0^\infty (\R^d)$ defined as follows.

\begin{definition} 
Let $\alpha: \N \rightarrow \N^d$ be a bijective map so that $\alpha(i) = (\alpha(i)^1, \ldots, \alpha(i)^d) \in \N^d$.
For any $V, W \in C^\infty_0(\mathbb{R}^d)$, we define the Fr\'echet metric by
\beq\label{eq:frechet1}
\rho_F(V,W)=\sum_{i \in {\N}} 2^{-i}\frac{\|D^{\alpha(i)} (V-W)\|_{2}}{1+\|D^{\alpha(i)}
(V-W)\|_{2}} .
\eeq
\end{definition}

\begin{proposition}\label{prop:compactness1}
The family $\mathcal{V}_R ( \mathcal{C}_s) \subset C_0^\infty (\overline{B}_R(0); \R)$
is compact with respect to the Fr\'echet metric \eqref{eq:frechet1}.
\end{proposition}


The proof follows from an application of the Ascoli Theorem \cite[Theorem I.28]{reed-simon1}.
We begin with a lemma.

\begin{lemma}\label{lemma:equicont}
The family $\mathcal{V}_R(\mathcal{C}_s)$ is equicontinuous in every derivative.
\end{lemma}

\begin{proof}
The family $\mathcal{V}_{R}(\mathcal{C}_s)$ is equicontinuous in each derivative if, for each $i \in \N$
and for any $\epsilon > 0$, there exists a $\delta > 0$ such that for all $V \in \mathcal{V}_R(\mathcal{C}_s)$, the condition $\|x-y\| \leq \delta$, for $x,y \in B_R(0)$, implies
$$
|D^{\alpha(i)} V(x)-D^{\alpha(i)} V(y)|\leq C_i \delta,
$$
for a constant $C_i$ depending only on $i$.
From the Sobolev Embedding Theorem, Theorem \ref{thm:sobolevEmbedding1}, we have
$$
\|V\|_{\infty}\leq \|V\|_{s,2}  \text{ for } s > \frac{d}{2} .
$$
which implies
$$
\|D^{\alpha(i)} V\|_{\infty}\leq \|V\|_{|\alpha (i)|+s, 2} ,  \text{ for } s > \frac{d}{2}.
$$
Since for all $V\in \mathcal{V}_{R}( \mathcal{C}_s)$, the $\|V\|_{s,2} \leq C_s$, for all $s \in \N$,
          we get a uniform bound on the $L^\infty$-norms of $D^{\alpha(i)} V$ which depends only on $|\alpha (i)|$.
By the Mean Value Theorem for $x,y\in B_R(0)$, we have
$$
|D^{\alpha (i)} V(x)-D^{\alpha (i)} V(y)| \leq \| \nabla (D^{\alpha(i)} V)\|_{\infty}\|x-y \| \leq
C_{|\alpha(i)|+1+ s} \|x-y\| ,
$$
for $s > \frac{d}{2}$ and where $C_{|\alpha(i)| +  s}$ is uniform with respect to $\mathcal{V}_{R}(\mathcal{C}_s)$. This establishes the equicontinuity of the family $\mathcal{V}_{R}(\mathcal{C}_s)$ in the ${\alpha(i)}^{th}$ derivative for each $i \in \N$.
\end{proof}


\noindent
{\bf Proof of Proposition \ref{prop:compactness1}}
\begin{proof}
1. Let $\{ V_j \} \subset \mathcal{V}_R( \mathcal{C}_s)$ be an arbitrary sequence. By equicontinuity, there exists a subsequence $\{ V_{j_i} \}$
that converges uniformly. By Lemma \ref{lemma:equicont}, for $i \in \N$ and multi-index $\alpha(i) \in \N^d$, the sequence $\{ D^{\alpha(i)} V_j \}$ is also equicontinuous and so has a uniformly convergent subsequence $\{V_{j_m (\alpha(i))}\}$.
So for each multi-index $\alpha(i)$, there is a subsequence $m \rightarrow j_m (\alpha(i))$ so that $\{ D^{\alpha(i)} V_{j_m(\alpha(i))} \}$ converges uniformly. Diagonalizing the sequence of subsequences $\{V_{j_m (\alpha(i))}\}_m$ yields the subsequence $\{V_{j_{m_\ell}(\alpha(i))}\}_\ell$ such that it and all the subsequences
$\{ D^{\alpha(i)} V_{j_{m_\ell} (\alpha(i))} \}_\ell$ converge uniformly on $B_R(0)$. Hence the limit potential $V \in C_0^\infty (\overline{B}_R(0); \R)$.

\noindent
2. We next prove that $V_{j_{m_\ell}} \to V$ in the Fr\'echet metric defined in \eqref{eq:frechet1}.
Because the series on the right in \eqref{eq:frechet1} converges, given $\epsilon > 0$, there exists $L > 0$, independent of $j$, so that
\beq\label{eq:upper-sum1}
\sum_{i=L+1}^\infty 2^{-i} \frac{\|D^{\alpha(i)} V_{j_{m_\ell}} - D^{\alpha(i)} V\|_{2}}{1+\|D^{\alpha(i)} V_{j_{m_\ell}}-D^{\alpha(i)} V\|_{2}} < \frac{\epsilon}{2} .
\eeq
To control the sum up to $L$, we note that
for each multi-index $\alpha(i)$, we can choose $J(i)$ such that $j>J(i)$ implies
$$
\|D^{\alpha(i)} V_{j_{m_\ell}}-D^{\alpha(i)} V \|_{2} < \frac{2^{i-1} \epsilon}{L+1} . 
$$
Setting $J >  \max \{ J(i) ~|~ i = 1, \ldots, L \}$,
it follows that for $j>J$
\beq\label{eq:lower-sum1}
\sum_{i=0}^L 2^{-i}\frac{\|D^{\alpha(i)} V_{j_{m_\ell}}-D^{\alpha(i)} V\|_{2}}{1+\|D^{\alpha(i)} V_{j_{m_\ell}}-D^{\alpha(i)} V\|_{2}} < \frac{\epsilon}{2}.
\eeq
Combining \eqref{eq:lower-sum1} and \eqref{eq:upper-sum1}, we obtain $\rho_F(V_{j_{m_\ell}}, V) < \epsilon$, for $j > J$.
We then note that subsequential limit point $V\in C^\infty_0(\overline{B}_R(0); \R)$
and that $\|D^{\alpha(i)} V\|_{2} < C_{\alpha(i)}$ for all multi-indices $\alpha(i)$.
So $V\in \mathcal{V}_{R} ( \mathcal{C}_s )$. This means that $\{V_n\}$ has a convergent subsequence so $\mathcal{V}_{R}( \mathcal{C}_s)$ is compact.
\end{proof}

\begin{lemma}\label{lemma:hurwitz1}
Suppose $\{ V_j \} \subset \mathcal{I}_R(V_0)$ converges to $V_\infty$ in the $\C^\infty$-Fr\'echet metric. Then $D_{p, V_\infty}(\lambda)$ has the same zeros with the same orders as $D_{p, V_0}(\lambda)$ for $\lambda \in \C$ and $p=1$ if $d=1$ and $p=2$ if $d=3$. Consequently, the potential
$V_\infty \in \mathcal{I}_R (V_0)$. Thus, the set $\mathcal{I}_R(V_0)$ is a closed subset of ${\mathcal{V}_R } ( \mathcal{C}_s)$ and hence compact.
\end{lemma}

\begin{proof}
 The second heat invariants satisfy $c_2(V_0) = c_2(V_j) = \int V^2_0 \neq$ so as $\lim c_2(V_j) = c_2(V_\infty)$, the limit potential
 $V_\infty$ is not identically zero. Since $V_\infty \in C_0^\infty (\R^d; \R)$, the corresponding Schr\"odinger operator has infinitely many resonances \cite{sabarreto-zworski1}. This means that the analytic function $D_{p,V_\infty}(\lambda)$ is not identically zero so we can apply the Hurwitz Theorem. The family of analytic functions $\{ D_{p, V_j}(\lambda) \}$ all have the same zeros, including order, as the function $D_{p, V_0}(\lambda)$. If $D_{p, V_\infty}(\lambda)$ did not have a zero at a zero of $D_{p, V_0}(\lambda)$, it would contradict Hurwitz Theorem. Similarly the order of the zero has to be the same as the order of the zero of $D_{p, V_0}(\lambda)$. Hence, the potential $V_\infty \in \mathcal{I}_R(V_0)$.
\end{proof}


\section{Compactness of the iso-resonant set $\mathcal{I}^m_R(V_0)$}\label{sec:less-reg-compactness}
\setcounter{equation}{0}

We discuss Proposition \ref{prop:main2} concerning less regular potentials in this section. For dimension $d=3$ and $V_0 \in H^{m}(\overline{B}_R(0); \R)$, we let $\mathcal{I}^m_R(V_0)$ denote the set of real potentials in $H^{m}(\overline{B}_R(0); \R)$ that are iso-resonant with $V_0$.
The first statement part of Proposition \ref{prop:main2} follows directly from \cite[Theorem 1.2]{smith-zworski1}.

As for the second part, suppose that $V_0 \in H^m(\overline{B}_R(0))$, for $m \geq 3$. Then, the resonances determine the wave trace through the Poisson formula. Assumption (WTE)
guarantees that the wave invariants satisfy $w_j(V) = w_j(V_0)$, for $j=2,3,\ldots, m+2$, for all $V \in {\mathcal{I}}^m_R(V_0)$.
Thus, for any $V \in \mathcal{I}_R^m(V_0)$, the heat invariants satisfy
\beq\label{eq:heat-inv1}
c_j(V) = c_j(V_0), ~~j=1, \ldots, m+2.
\eeq
It is a consequence of the formula \eqref{eq:wave-tr-coef1}, the identity \eqref{eq:wave-to-heat-coef1}, \eqref{eq:heat-inv1}, and the induction of section \ref{sec:sobolev-bounds1}, that for $j=0, \ldots, m$, we have uniform bounds on the Sobolev norms across the iso-resonant class:
\beq\label{sobolev-less-reg1}
\| V \|_{j,2} \leq C_j, ~~\forall ~~ V \in \mathcal{I}_R^m(V_0) .
\eeq
As in the proof of equicontinuity in Lemma \ref{lemma:equicont}, the derivatives of $V$ are uniformly bounded in the $L^\infty$-norm up to and including order $m-2$ for $d=3$. Finally, the same argument shows uniform equicontinuity of the derivatives of the potentials up to order $m-3$. The compactness of ${\mathcal{I}}^m_R(V_0) \subset H^{m-3}(\overline{B}_R(0))$ now follows by the same argument as in section \ref{sec:compactness}.



 \section{Appendix: Proof of the uniform Sobolev bounds}\label{sec:sobolev-bounds1}
\setcounter{equation}{0}

In this appendix, we present the details of the uniform Sobolev bounds on the iso-resonant class of potentials $\mathcal{I}_R (V_0)$ presented in Theorem \ref{thm:uniform-sobolev1}. We use standard notation $H^{s,p}(\Omega)$, for $\Omega \subset \R^d$, an open set (with a boundary satisfying the cone condition, see Theorem \ref{thm:sobolevEmbedding1}), for the space of functions with $s$-weak-derivatives in $L^p(\Omega)$. The norm is denoted by $\| \cdot \|_{s,p}$.

We recall from \eqref{eq:indexset1} the index set of $k$-tuples $\mathcal{A}_{j,k}$, defined for $j \geq 3$ and
$3 \leq k \leq j$ by:
\bea\label{eq:indexset2}
\mathcal{A}_{j,k} &=& \left\{ \alpha = ( \alpha^1, \ldots , \alpha^k) \left| \begin{array}{l}
   \alpha^m = ( \alpha_1^m, \ldots, \alpha_d^m) \in \N_0^d \\
   | \alpha^m| \leq j-k \\
   \sum_{m=1}^k  | \alpha^m| = 2(j-k) \\
   \sum_{m=1}^k \alpha_{\ell}^m \in 2 \N, ~~\forall ~ \ell =1, \ldots, d.
    \end{array} \right.
\right\} .
\eea

From Proposition \ref{prop:wave-coef1} and \eqref{eq:wave-tr-coef2}, the Sobolev norm $ \|V\|_{j-2,2}^2$, for $j \geq 3$, is bounded above as
\beq\label{eq:sobolev-sum1}
 \|V\|_{j-2,2}^2 \leq C_j \left( 1 +  \sum_{k=3}^j \sum_{\alpha \in \mathcal{A}_{j,k}} \int_{\R^d} | D^{\alpha^1}(V)D^{\alpha^2}(V)\cdots D^{\alpha^k}(V)| ~d^dx \right) ,
\eeq
where the constant $C_j > 0$ is independent of our choice of $V$ from the iso-resonant set.
Proceeding inductively, we must find an upper bound on the sum \eqref{eq:sobolev-sum1}
for the $H^{j-2,2}$-norm of $V$ in terms of the lower Sobolev norms of $V$.




\subsection{The one-dimensional case}\label{subsec:one-dim1}

We first discuss the case $d=1$ for which the bounds are easier to obtain.
We begin with a simple and useful bound.

\begin{lemma} \label{Dimension 1 L-infty bounds}
For any $u\in C^1_0(\mathbb{R})$, we have
$$
\|u\|_\infty\leq C \|u\|_{1,2} ,
$$
where the constant depends on the support of $u$.
\end{lemma}


The first step of an inductive proof of the uniform Sobolev bounds is the following proposition.

\begin{proposition}\label{prop:sobolev-bd1d}
Let $d=1$, $j \geq 3$, and suppose that $\|V\|_{j-3,2}\leq M$.
We then have
\beq\label{eq:induction1d1}
\int_\R ~|D^{\alpha^1}(V)D^{\alpha^2}(V)\cdots D^{\alpha^k}(V)|\leq C_j ,
\eeq
where $C_j > 0$ depends on $M$ and $j$.
\end{proposition}

\begin{proof}
 1. Since we are in one dimensions, we write $\alpha^m$ for $\alpha_1^m$.
 We use the bounds on the order of the $D^{\alpha^i}(V)$-terms to conclude: 1) since $k \geq 3$, we have $\alpha^m \leq j - 3$, 2) we have the constraint $\sum_{m=1}^k  \alpha^m \leq  2(j-3)$, and 3) there at at most two terms with order $(j-3)$, since
 $$
 \sum_{m=1}^k \alpha^m = 2(j-k) \leq 2(j-3) .
 $$
These restrictions, together with Lemma \ref{Dimension 1 L-infty bounds}, will then allow us to get the desired bounds as follows:


\noindent
2. \textbf{Case 1:}  We assume that the product contains no terms of $j-3$.
Then, using Lemma \ref{Dimension 1 L-infty bounds} for each $i$,
\beq\label{eq:case1-prelimd1}
|D^{\alpha^i}(V)| \leq C \|D^{\alpha^i}(V)\|_{1,2} \leq  C \|V\|_{j-3,2} .
\eeq
This gives
\beq\label{eq:case1-d1}
\int_{\R} |D^{\alpha^1}(V)D^{\alpha^2}(V)\cdots D^{\alpha^k}(V)|\leq C^k M^k\leq C^j M^j ,
\eeq
where we assume that $C, M\geq max(1,|B_R(0)|)$.

\noindent
3. \textbf{Case 2:} We assume that the product contains one term of order $j-3$, say $\alpha^1$.
Using the results from Case 1, together with the H\"{o}lder inequality, we have
\bea\label{eq:case2-d1}
\int_{\R} |D^{\alpha^1}(V) D^{\alpha^2}(V) \cdots D^{\alpha^k}(V)| & \leq & C^{k-1} M^{k-1} \int_\R |D^{\alpha^1}(V)| \nonumber \\
 & \leq & C^{k-1} M^{k-1} |B_R(0)|  \|V\|_{j-3,2} \nonumber \\
  & \leq  & C^j M^{j} .
\eea

\noindent
3. \textbf{Case 3:} We assume the product contains two terms of order $j-3$.
Using the results of Case 1 and the H\"{o}lder inequality, we obtain
\bea\label{eq:case3-d1}
\int_\R |D^{\alpha^1}(V) D^{\alpha^2}(V) \cdots D^{\alpha^k}(V)| & \leq & C^{k-2} M^{k-2} \int_\R |D^{\alpha^1}(V) D^{\alpha^2}(V)| \nonumber \\
 & \leq & C^{k-2} M^{k-2}   \|V\|^2_{j-3,2}\leq C^j M^{j}.
\eea
\end{proof}

It follows from the bound in Proposition \ref{prop:sobolev-bd1d} and \eqref{eq:case1-d1}--\eqref{eq:case3-d1}
that
\beq
\|V\|_{j-2,2}^2 \leq C_j,
\eeq
where the constant depends on $M$. Consequently, the constant is uniform over all $V \in \mathcal{I}_R (V_0)$.

\subsection{The $d \geq 3$ dimensional case}\label{subsec:odd-dim1}

For $d \geq 3$, we follow the idea of the proof given by Donnelly  \cite{donnelly}. This requires reordering the $D^{\alpha^i}(V)$ terms in the integral of \eqref{eq:sobolev-sum1} according to their order $|\alpha^i|$.
For fixed $k$, we write the integrand of the integral in \eqref{eq:sobolev-sum1} as
\beq\label{eq:ordering1}
T =: D^{\alpha^1}(V) D^{\alpha^2}(V) \cdots D^{\alpha^{\ell}}(V) D^{\alpha^{\ell +1}}(V) \cdots D^{\alpha^k}(V) ,
\eeq
where the ordering is chosen such that
\bea\label{eq:ordering2}
i \leq \ell & \Rightarrow	&  d > 2(j- |\alpha^i| - 3) \\
i  >   \ell & \Rightarrow   &  d \leq 2( j - |\alpha^i| - 3 ).
\eea
We will bound the integral in \eqref{eq:sobolev-sum1} using the generalized H{\"o}lder's inequality \eqref{eq:gen-holder1} and the Sobolev Embedding Theorem, Theorem \ref{thm:sobolevEmbedding1}.
 The conditions on $|\alpha^i|$ determine which case of the Sobolev Embedding Theorem, Theorem \ref{thm:sobolevEmbedding1}, for $p=2$ and
       $k=(j- |\alpha^i| -3)$ is appropriate.

\begin{proposition}\label{Donnelly:4.6} \cite[Lemma 4.6]{donnelly}
If $d\geq 3, j>\frac{d}{2}+1$,  and $ \|V\|_{j-3,2}\leq C_1$,  then
\beq\label{eq:multiT1}
\int_{\R^d} ~ |D^{\alpha^1}(V) D^{\alpha^2}(V) \cdots D^{\alpha^{\ell}}(V) D^{\alpha^{\ell +1}}(V) \cdots D^{\alpha^k}(V)| \leq  C_2\left(1+\|V\|_{j-2,2}^\beta\right) ,
\eeq
where $\beta<2$ and $C_2$ depends on $C_1$.
\end{proposition}

\begin{proof}
1. We will look at the possible values of $\ell$ and for each case the general strategy will be to use the generalized H\"{o}lder's inequality \eqref{eq:gen-holder1}to show
\beq\label{eq:sobolev-sum2}
\int_{\R^d} ~ |T| \leq C \prod \limits_{i=1}^k \|D^{\alpha^i}(V)\|_{r_i} ,
\eeq
with $\sum\limits_{i=1}^k \frac{1}{r_i}=1$. We recall that the integrand $T$ in \eqref{eq:sobolev-sum1} is ordered as in \eqref{eq:ordering1} with $\ell$ determined as in \eqref{eq:ordering2}.
For the factors with $i \geq \ell+1$,
so that $2(j- |\alpha^i|-3) > d$, we have
\beq\label{eq:sobolev1-ileql}
\|D^{\alpha^i} (V) \|_\infty\leq C\|V\| _{j-3,2},
\eeq
and when $i$ is such that $2(j- | \alpha^i| - 3 ) = d$, for any $r_i$ with $2\leq r_i <\infty$, we have
\beq\label{eq:sobolev2-ileql}
\| D^{\alpha_i}(V)\|_{r_i} \leq C \|V\| _{j-3,2} .
\eeq
These two inequalities result in the bound
\beq\label{eq:firstbd1}
\int_{\R^d} ~ |T| \leq C \prod\limits_{i=1}^k \|D^{\alpha^i}(V)\|_{r_i}  \leq  \tilde{C} \|V\| _{j-3,2}^{k-\ell} \prod\limits_{i=1}^\ell \|D^{\alpha^i}(V)\|_{r_i} .
\eeq
The remainder of the proof is devoted to showing that for $1 \leq i \leq \ell$, we can choose $r_i$ in \eqref{eq:sobolev-sum2} in order to apply the appropriate Sobolev inequality in Theorem \eqref{thm:sobolevEmbedding1}. We first note that when $\ell=0$ the estimate holds for $\beta=0$ using the above method.


\noindent
3. \textbf{Case 1:} For $\ell=1$, we have $d>2(j-|\alpha^1|-3)$, so setting
\[r_1=\frac{2d}{d-2(j-|\alpha^1|-3)}\]
yields
$$
\|D^{\alpha^1}(V)\|_{r_1}\leq C\|D^{\alpha^i}(V)\|_{j-|\alpha^1|-3,2}\leq C\|V\| _{j-3,2},
$$
by the Sobolev inequality in Theorem \eqref{thm:sobolevEmbedding1}. The only condition on $j$ is $2\leq \frac{2d}{d-2(j-|\alpha^1|-3)}$ or $j-3\geq  |\alpha^1|$ which is true for every $\alpha^i$ and $j$ as in \eqref{eq:indexset1}.
Since $\frac{1}{r_1}\leq \frac{1}{2}$, we can choose the remaining $r_i$'s to meet the condition
$\sum\limits_{i=1}^k \frac{1}{r_i}=1$.  So for $\ell=1$ we have the bound with $\beta=0$.

\noindent
4. \textbf{Case 2:} For $\ell=2$, we argue as follows.
If $|\alpha^1|$ and $|\alpha^2|$ are such that $r_1$ and $r_2$ (as chosen in case $\ell=1$) satisfy
\[\frac{1}{r_1}+\frac{1}{r_2}<1\]
then we proceed as in the case $\ell=1$ and apply the generalized H\"{o}lder's inequality to get the result with $\beta =0$.
Now assume $\frac{1}{r_1}+\frac{1}{r_2}=1$. Since $|\alpha^i|\leq j-3$, this implies $|\alpha^1|=|\alpha^2|=j-3$ and thus $r_1=r_2=2$.  We may then apply the generalized H\"{o}lder's inequality to get:
\[\int_{\R^d} |T|\leq C \|D^{\alpha^1}(V)\|_{r_1+\varepsilon}\prod\limits_{i=2}^k \|D^{\alpha^i}(V)\|_{r_i}, \]
where $\varepsilon>0$ and $r_i$ for $i \geq 3$ are chosen to satisfy the H\"{o}lder condition.  Furthermore if we choose $\varepsilon$ such that $r_1+\varepsilon<\frac{2d}{d-2}$ then the general Sobolev inequality in Theorem \ref{thm:sobolevEmbedding1} gives that
\[\|D^{\alpha^1}(V)\|_{r_1+\varepsilon_i}\leq C_1\|D^{\alpha^1}(V)\|_{1,2}\leq C_2\|V\|_{j-2,2}  , \]
so we get the result with $\beta=1$.

\noindent
5.  \textbf{Case 3:} We now consider $\ell\geq 3$ and we suppose that $d>2(j-|\alpha^i|-2)$ for, say, $i=1,2$.
Let $r_i$ be as in cases 1 and 2 and set
\[
s_i=\frac{2d}{d-2(j-|\alpha^i| -2)} .
\]
The standard $L^p$ interpolation \eqref{eq:interpolate1} estimate allows us to conclude that for any $0<\varepsilon_i<1$, there exists a $0<\beta_i<1$ such that
\beq
\|D^{\alpha^i}(V)\|_{r_i+\varepsilon_i} \leq \|D^{\alpha^i}(V)\|_{r_i}^{\beta_i}\|D^{\alpha^i}(V)\|_{s_i}^{1-\beta_i} .
\eeq
Using the generalized H\"{o}lder's inequality we have
\bea\label{eq:sobolev-sum3}
\int_{\R^d} |T|&\leq & C \|D^{\alpha^1}(V)\|_{r_1+\varepsilon_1}\|D^{\alpha^2}(V)\|_{r_2+\varepsilon_2}\prod\limits_{i=3}^k \|D^{\alpha^i}(V)\|_{r_i} \nonumber  \\
& \leq & \|D^{\alpha^1}(V)\|_{r_1}^{\beta_1}\|D^{\alpha^1}(V)\|_{s_1}^{1-\beta_1}\|D^{\alpha^2}(V)\|_{r_2}^{\beta_2}\|D^{\alpha^2}(V)\|_{s_2}^{1-\beta_2}\prod\limits_{i=3}^k \|D^{\alpha^i}(V)\|_{j-|\alpha^i| -3,2} \nonumber  \\
&\leq & C\|V\|_{j-3,2}^{\beta_1}\|D^{\alpha^1}(V)\|_{s_1}^{1-\beta_1}\|V\|_{j-3}^{\beta_1}\|D^{\alpha^2}(V)\|_{s_2}^{1-\beta_2}\prod\limits_{i=3}^k \|V\|_{j -3,2} \nonumber  \\
&\leq  & C\|D^{\alpha^1}(V)\|_{s_1}^{1-\beta_1}\|D^{\alpha^2}(V)\|_{s_2}^{1-\beta_2} \nonumber  \\
&\leq  & C\|V\|_{j-2,2}^\beta,
\eea
where $\beta<2$. The index $r_i$ may be chosen arbitrarily for $i>l$ and as in case 1 for $i\leq l$. Consequently, in order to satisfy the H\"{o}lder condition we require:
\beq\label{eq:holder1}
\frac{1}{r_1+\varepsilon_1}+\frac{1}{r_2+\varepsilon_2}+\sum\limits_{i=3}^l \frac{1}{r_i}<1  .
\eeq
For sufficiently large $\varepsilon_1,\varepsilon_2 > 1$, this is implied by
$$
\frac{1}{s_1}+\frac{1}{s_2}+\sum\limits_{i=3}^l \frac{1}{r_i}<1  .
$$
Substituting for $s_i$ and $r_i$ gives
\beq\label{eq:holder5}
\sum\limits_{i=1}^2\frac{d-2(j-|\alpha^i|-2)}{2d}+\sum\limits_{i=3}^l\frac{d-2(j-|\alpha^i|-3)}{2d}<1 ,
\eeq
which may be rewritten as
$$
(d-2j-6)l+2\sum\limits_{i=1}^l |\alpha^i|<2d+4  .
$$
Because $\sum\limits_{i=1}^l |\alpha^i|\leq 2(j-k)$, it is sufficient to show
$$
(d-2j-6)l+4(j-k)<2d+4 .
$$
Using assumption $l\geq 3$ lets us rewrite the inequality as
$$
\frac{d}{2}+3-\frac{2k-4}{l-2}<j  .
$$
Then $k\geq l\geq 3$ gives $\frac{2k-4}{l-2}\geq\frac{2k-4}{k-2}=2$. Consequently, the condition on the indices for the generalized H\"older inequality  \eqref{eq:holder1} is satisfied provided
$$
\frac{d}{2}+1<j  ,
$$
which is an assumption on $j$ and $d \geq 3$ in Proposition \ref{prop:sobolev-bd1d}. 

\noindent
6. \textbf{Case 4:} We now consider $\ell \geq 3$ and we suppose the complementary situation to case 3: $d\leq 2(j-|\alpha^i| -2)$ for, say, $i=1,2$.
For $2\leq s< \infty$, we have the embedding
$$
\|D^{\alpha^i}(V)\|_s \leq C \|D^{\alpha^i}(V)\|_{j-|\alpha^i| -2,2} \leq  C \|V\|_{j-2,2}.
$$
Then, standard $L^p$ interpolation given in Theorem \ref{thm:interpolation1}, gives for $2< t<s$
$$
\|D^{\alpha^i}(V)\|_t\leq \|D^{\alpha^i}(V)\|_s^{\beta_i}\|D^{\alpha^i}(V)\|_2^{1-\beta_i}.
$$
We may take $t$ to be arbitrarily large reducing the H\"{o}lder condition to
$$
\sum\limits_{i=3}^l\frac{1}{r_i}<1.
$$
If $\ell=3$, the condition is met as $r_3\geq 2$, so we assume $l\geq 4$.  Substituting for $r_i$ and rewriting the inequality we get
$$
(\ell-2)(d-2j+6)+2\sum\limits_{i=3}^\ell |\alpha^i|<2d.
$$
Using the inequality $\sum\limits_{i=3}^l |\alpha^i|\leq\sum\limits_{i=1}^k |\alpha^i|\leq 2(j-k)$ gives the sufficient condition
$$
(l-2)(d-2j+6)+4(j-k)<2d ,
$$
which can be recast as
$$
(l-4)d+6(l-2)-4k<(2l-8)j.
$$
If $\ell=4$, then the inequality reduces to $12-4k<0$ which always holds as $4= \ell\leq k$.  For $\ell\geq 5$, we rewrite the inequality as
$$
\frac{d}{2}+3\frac{l-4}{l-4}+3\frac{2}{l-4}-\frac{2k}{l-4}<j ,
$$
which reduces to
$$
\frac{d}{2}+3-\frac{(2k-6)}{\ell-4}< j.
$$
Since $k\geq l \geq 5$, it is sufficient for the following inequality to hold:
$$
\frac{d}{2}+3-\frac{(2k-6)}{k-4}< j .
$$
With $k \geq 5$, this is satisfied if
$$
\frac{d}{2}+1-\frac{4}{k-4}< j ,
$$
which is guaranteed by the hypothesis $\frac{d}{2}+1< j$.
\end{proof}

\section{Appendix: Various estimates}\label{sec:appendix-est1}
\setcounter{equation}{0}

We summarize various estimates necessary for the proofs. Complete descriptions for the material in sections \ref{subsec:resolvent-bds1} and \ref{subsec:sing-values1} may be found in, for example, \cite{dyatlov-zworski2016},  and for section \ref{subsec:inequalities1} in, for example, \cite{evans1} and \cite{adams-fournier1}.


\subsection{Resolvent bounds}\label{subsec:resolvent-bds1}

The kernels of the resolvent $R_0(\lambda) = ( H_0 - \lambda^2)^{-1}$ in dimensions $d=1$ and $d=3$ are given by
\bea\label{eq:kernel1}
R_0(x-y;\lambda)& = & \frac{i}{2 \lambda} e^{i \lambda \|x-y\|} , ~~~d=1 , \\
R_0(x-y;\lambda)& = & \frac{1}{4 \pi^2} \frac{e^{i \lambda \|x-y\|}}{\| x-y \|} , ~~~d=3.
\eea
Let $\rho \in L_0^\infty (\R^d)$ be a compactly supported function.  be the characteristic function on a ball of radius $R>0$.
For all $\lambda \in \C \backslash \{ 0 \}$, the free resolvent in $d=1$ satisfies the bound:
\beq\label{eq:resolvent-est1}
\| \rho  R_0(\lambda) \rho \| \leq C_\rho \frac{1}{| \lambda|} e^{- \alpha_\rho \Im \lambda } .
\eeq
For all $\lambda \in \C$, the free resolvent in $d=3$ satisfies the bound:
\beq\label{eq:resolvent-est2}
\| \rho R_0(\lambda) \rho \| \leq D_\rho  e^{- \beta_\rho  \Im \lambda } .
\eeq
The finite positive constants $C_\rho, D_\rho, \alpha_\rho, \beta_\rho > 0$ depend on $\rho$.


\subsection{Singular value estimates}\label{subsec:sing-values1}

If $\rho \geq 0$ is a compactly supported $C^2$-function, then
\beq\label{eq:sing-value1}
\mu_j(\rho R_0(\lambda) \rho) \leq \frac{C}{j^{\frac{2}{d}}} .
\eeq
Recall that for $V \in C_0^\infty (B_R(0);\R)$, the operator $K_V(\lambda) := V R_0(\lambda) \chi_R$, where $\chi_R$
is the characteristic function on $B_R(0)$. We then have
\beq\label{eq:sing-value2}
\mu_j(K_V(\lambda)) \leq \frac{C(V)}{j^{\frac{2}{3}}} ,
\eeq
where $C(V)$ depends only of $\| V \|_\infty$ and is locally uniformly bounded in $\lambda$.
A consequence of \eqref{eq:sing-value2} is that $K_V(\lambda)$ is in the Hilbert Schmidt class.


\subsection{Inequalities}\label{subsec:inequalities1}

Section \ref{sec:sobolev-bounds1} makes repeated use of the following results.

\vspace{.1in}

\noindent
\emph{Generalized H\"older Inequality}

\vspace{.1in}
\begin{theorem}\cite[Appendix B.2.g.]{evans1}
Let $\Omega \subset $ be an open subset of $\mathbb{R}^d$. For indices $1 \leq p_1, \ldots, p_m \leq \infty$ satisfying
 $$
 \sum_{j=1}^m \frac{1}{p_j} = 1,
 $$
for any $u_k \in L^{p_k} (\Omega)$, we have
\beq\label{eq:gen-holder1}
\int_\Omega ~|u_1 \cdots u_m| \leq \prod_{j=1}^m \| u_k \|_{p_j} .
\eeq
\end{theorem}

\vspace{.1in}

\noindent
\emph{General Sobolev Inequality}

\vspace{.1in}
The Sobolev Embedding Theorem and corresponding inequalities are used through section \ref{sec:sobolev-bounds1} and \ref{sec:less-reg-compactness}. The general theorem holds for domains $\Omega \subset \R^d$ with the cone conditions \cite[Definition 4.6]{adams-fournier1}. We state only the cases used in the present paper.

\vspace{.1in}
\begin{theorem}\cite[Theorem 4.12]{adams-fournier1} \cite[Theorem 6, section 5.6]{evans1}\label{thm:sobolevEmbedding1}
Let $\Omega \subset \mathbb{R}^d$ be an open set with the cone condition and let $j \geq 0$ and $m \geq 1$ be integers. 
\begin{enumerate}

\item Let $u\in H^{k,p} (\Omega)$.
If $0 < k< \frac{d}{p}$, then $u\in
L^q( \Omega)$, where the indices satisfy $\frac{1}{q}=\frac{1}{p}-\frac{k}{d}$, and
$$
\| u \|_{q} \leq C \| u \|_{k,p},
$$
for a finite constant $C > 0$ depending on $k,p, d$ and $\Omega$.

\item For any $u\in H^{j+m,2}(\Omega)$, with $m > \frac{d}{2}$,
we have
\begin{equation}
\|u\|_{{j,q}}  \leq C \|u\|_{{j+m,2}},
\end{equation}
for $2 \leq q \leq \infty$ and a finite constant $C>0$ depending on $j, m,d$ and $\Omega$.

\end{enumerate}
\end{theorem}

\vspace{.1in}

%
%

\vspace{.2in}

\noindent
\emph{$L^p$-interpolation Inequality}

\vspace{.1in}

\begin{theorem}\cite[Appendix B.2.h.]{evans1}\label{thm:interpolation1}
Let $\Omega \subset \mathbb{R}^d$ be an open bounded subset of $\R^d$. If $0 < \alpha < 1$ and $1 \leq r < p < s$, for all $u \in L^p (\Omega)$,
we have,
\beq
\|u\|_p \leq \| u \|_r^\alpha ~ \| u \|_s^{1 - \alpha} ,
\eeq
provided the indices satisfy
$$
\frac{1}{p} = \frac{\alpha}{r} + \frac{1-\alpha}{s} .
$$
\end{theorem}


\end{document}